\documentclass[12pt, reqno, a4paper]{amsart}

\usepackage{amssymb, amsmath, enumerate, amsfonts, amsthm, mathrsfs, url, bm}

\usepackage{xcolor}  	
\usepackage[backref]{hyperref}
\hypersetup{
	colorlinks,
    linkcolor={blue!60!black},
    citecolor={blue!60!black},
    urlcolor={red!60!black}
}

\usepackage{color}

\advance \topmargin by -\headheight
\advance \topmargin by -\headsep
\evensidemargin \oddsidemargin
\numberwithin{equation}{section}

\newtheorem{theorem}{Theorem}[section]
\newtheorem{lemma}[theorem]{Lemma}    
\newtheorem{corollary}[theorem]{Corollary}

  \theoremstyle{remark}

\theoremstyle{definition}
\newtheorem{definition}[theorem]{Definition}
\newtheorem{remark}[theorem]{Remark}

\newtheorem*{question*}{Question}

\renewcommand{\geq}{\geqslant}
\renewcommand{\leq}{\leqslant}

\newcommand{\set}[1]{\left\{#1\right\}}

\newcommand{\abs}[1]{\left| #1\right|}

\newcommand{\Bigabs}[1]{\Bigl| #1 \Bigr|}
\newcommand{\biggabs}[1]{\biggl| #1 \biggr|}

\newcommand{\ceil}[1]{\left\lceil #1 \right\rceil}

\newcommand{\floor}[1]{\left\lfloor #1 \right\rfloor}
\newcommand{\brac}[1]{\left( #1 \right)}
\newcommand{\bigbrac}[1]{\bigl( #1 \bigr)}

\newcommand{\norm}[1]{\left\| #1\right\|}

\newcommand{\recip}[1]{\frac{1}{#1}}
\newcommand{\trecip}[1]{\tfrac{1}{#1}}

\newcommand{\vh}{\mathbf{h}}

\newcommand{\N}{\mathbb{N}}
\newcommand{\Z}{\mathbb{Z}}

\newcommand{\R}{\mathbb{R}}
\newcommand{\C}{\mathbb{C}}

\newcommand{\F}{\mathbb{F}}

\newcommand{\E}{\mathbb{E}}

\newcommand{\supp}{\mathrm{supp}}

\newcommand{\codim}{\mathrm{codim}\,}

\newcommand{\eps}{\varepsilon}

\makeatletter
\let\@@pmod\pmod
\DeclareRobustCommand{\pmod}{\@ifstar\@pmods\@@pmod}
\def\@pmods#1{\mkern4mu({\operator@font mod}\mkern 6mu#1)}
\makeatother

\begin{document}

\title{On the Ramsey number of the Brauer Configuration}

\author{Jonathan Chapman}
\address{Department of Mathematics\\ University of Manchester\\ Oxford Road\\ Manchester\\ M13 9PL\\ UK}
\email{jonathan.chapman@manchester.ac.uk}
\author{Sean Prendiville}
\address{Department of Mathematics and Statistics
	\\ Lancaster University\\  UK}
\email{s.prendiville@lancaster.ac.uk.}

%\thanks{}

%\subjclass[2010]{11B30, 11B25}

\date{\today}

\begin{abstract}
	We obtain a double exponential bound in Brauer's generalisation of  van der Waerden's theorem, which concerns progressions with the same colour as their common difference. Such a result has been obtained independently and in much greater   generality by Sanders. Using Gowers' local inverse theorem, our bound is quintuple exponential in the length of the progression.  We  refine this bound in the colour aspect for three-term progressions, and combine our arguments with an insight of Lefmann to obtain analogous bounds for the Ramsey numbers of certain nonlinear quadratic equations. 
\end{abstract}

\maketitle

\setcounter{tocdepth}{1}
\tableofcontents

\section{Introduction}

Schur's theorem states that in any partition of the positive integers into finitely many pieces, at least one part contains a solution to the equation $x+y = z$.  By a theorem of van der Waerden, the same is true for the equation of three-term arithmetic progressions $x+y = 2z$. A common generalisation of these theorems due to Brauer states that, in any finite colouring of the positive integers, there is a monochromatic arithmetic progression of length $k$ with the same colour as its common difference.

There is a finitary analogue of these results, asserting that the same holds for colourings of the interval $\set{1, 2, \dots, N}$, provided that $N$ is sufficiently large in terms of the number of colours. Determining the minimal such number $N$ (the \emph{Ramsey} or \emph{Rado} number of the system) has received much attention for arithmetic progressions, and a celebrated breakthrough of Shelah \cite{ShelahPrimitive} showed that these numbers (\emph{van der Waerden numbers}) are primitive recursive.  One spectacular consequence of Gowers' work on Szemer\'edi's theorem \cite{GowersSzem} is a bound on van der Waerden numbers which is quintuple exponential in terms of the length of the progression, and double exponential in terms of the number of pieces of the partition.

Using Gowers' local inverse theorem for the uniformity norms, we obtain a bound for the Ramsey number of Brauer configurations which is comparable to that obtained for arithmetic progressions.

\begin{theorem}[Ramsey bound for Brauer configurations]\label{weak brauer theorem}
	There exists an absolute constant $C=C(k)$ such that if $r \geqslant 2$ and $N\geqslant \exp\exp(r^{C})$, then any $r$-colouring of $\{1,2,...,N\}$ yields a monochromatic $k$-term progression  which is the same colour as its common difference.  Moreover, it suffices to assume that
		\begin{equation}\label{mainBound}
	N\geqslant 2^{2^{r^{2^{2^{k+10}}}}}.
	\end{equation}
\end{theorem}

A double exponential bound has been obtained independently and in maximal  generality by Sanders \cite{SandersBootstrapping}, who bounds the Ramsey number of an \emph{arbitrary} system of linear equations with this colouring property, often termed \emph{partition regularity}.\footnote{A system of equations is said to be \emph{partition regular} (over $\N$) if any finite colouring of $\N$ yields a monochromatic solution to the system.} Our results and those of Sanders are the first quantitatively effective bounds for configurations lacking translation invariance and of `true complexity' greater than one (see \cite{GowersWolfTrue} for further explanation).  

Gowers \cite{GowersSzem} obtains the bound \eqref{mainBound} for progressions of length $k+1$, this being the appropriate analogue of the $(k+1)$-point Brauer configuration in Theorem \ref{weak brauer theorem}.  For a four-point Brauer configuration, we improve the exponent of $r$ on combining our method with an energy-increment argument of Green and Tao \cite{GT09}.

\begin{theorem}[Improved bound for four-point Brauer configurations]\label{improved bound}
\sloppy There exists an absolute constant $C$ such that if  $N\geqslant  \exp\exp(Cr\log^2r)$, then in any $r$-colouring of $\set{1,2, \dots, N}$ there exists a monochromatic three-term progression with the same colour as its common difference.
\end{theorem}

Due to an insight of Lefmann \cite{lefmann} we are able to use (a variant of) Theorem \ref{weak brauer theorem} to bound the Ramsey number of certain partition regular \emph{nonlinear} equations.
\begin{theorem}\label{intro lefmann}
Let $a_1, \dots, a_s \in \Z\setminus\set{0}$ satisfy the following:
\begin{enumerate}[\upshape (i)]
		\item there exists a non-empty set $I\subset[s]$ such that $\sum_{i\in I}a_{i}=0$;
		\item the system
		\begin{equation*}
		x_0^{2}\sum_{i\notin I}a_{i}+ \sum_{i\in I}a_{i}x_{i}^{2}=
		\sum_{i\in I}a_{i}x_{i}=0.
		\end{equation*}
		has a rational solution with $x_0 \neq 0$.
	\end{enumerate}
Then there exists an absolute constant $C = C(a_1, \dots, a_s)$ such that for $r \geq 2$ and $N \geq \exp \exp (r^C)$, any $r$-colouring of $\set{1, 2, \dots, N}$ yields a monochromatic solution to the diagonal quadric
$$
a_1 x_1^2 + \dots + a_s x_s^2 = 0.
$$
\end{theorem}
\subsection*{Previous work}
Hitherto, little is recorded regarding the Ramsey number of general partition regular systems.  Cwalina--Schoen \cite{CwalinaSchoen} observe that one can use Gowers' bounds \cite{GowersSzem} in Szemer\'edi's theorem to obtain a bound which is tower in nature, of height proportional to $5r$.  Gowers' methods are well suited to delivering double exponential bounds for so-called translation invariant systems (such as arithmetic progressions), but such systems are far from typical.   In Cwalina--Schoen \cite{CwalinaSchoen}, Fourier-analytic arguments are adapted  to give an exponential bound on the Ramsey number of a single partition regular equation.   The first author \cite{chapman} has shown how multiplicatively syndetic sets allow one to reduce the tower height to $(1+o(1))r$ for the four-point Brauer configuration 
\begin{equation}\label{four point}
x,\ x+d,\ x+2d,\ d.
\end{equation}
\subsection*{Our method}
The approach underlying Theorem \ref{weak brauer theorem} generalises that of Roth \cite{roth} and Gowers \cite{GowersSzem}. Given a subset $A$ of $[N] := \set{1, 2, \dots, N}$ of density $\delta$, Roth uses a density increment procedure to locate a subprogression $P = a + q\cdot [M]$ of length $M \geq N^{\exp(-C/\delta)}$ where $A\cap P$ has density at least $\delta$ and is `Fourier uniform', in the sense that all but its trivial Fourier coefficients are small. An application of Fourier analysis (in the form of the circle method) shows that such sets possess of order $\delta^3 M^2$ three-term progressions. This yields a non-trivial three-term progression provided that $N \geq \exp\exp(C/\delta)$. 

The above application of the circle method relies crucially on the translation-dilation invariance of three-term progressions, so that the number of configurations in $P$ is the same as that in $[M]$.  Unfortunately, the Brauer configuration \eqref{four point} is not translation invariant. To overcome the lack of translation-invariance, given a colouring $A_1 \cup \dots \cup A_r = [N]$, we use Gowers' local inverse theorem for the uniformity norms \cite{GowersSzem} to run a density increment procedure with respect to the \emph{maximal translate density}
$$
\sum_{i=1}^r \max_a\frac{|A_i\cap (a + q\cdot [M])|}{M}.
$$ 
 This outputs (see Lemma \ref{thmUniBrau}) a homogeneous progression $q \cdot [M]$ such that for each colour class $A_i$ there is a translate $a_i + q\cdot [M]$ on which $A_i$ achieves its maximal translate density and on which $A_i$ is suitably uniform.  For the four-point Brauer configuration \eqref{four point}, the correct notion of uniformity is \emph{quadratic uniformity} (as measured by the Gowers $U^3$-norm).  
 
 Write $\alpha_i$ for the density of $A_i$ on the maximal translate $a_i + q\cdot[M]$ and $\beta_i$ for its density on the homogeneous progression $q\cdot [M]$.  An application of quadratic Fourier analysis shows that the number of four-point Brauer configurations \eqref{four point} satisfying
 $$
 \set{x, x+d, x+2d} \subset A_i \cap (a_i + q\cdot [M]) \quad \text{and}\quad d \in A_i \cap q\cdot [M]
 $$ 
 is of order $\alpha_i^3\beta_i M^2$.   By the pigeonhole principle, some colour class has $\beta_i  \geq 1/r$, which also implies that $\alpha_i \geq 1/r$, so we deduce that some colour class contains at least $r^{-4} M^2$ Brauer configurations.  Unravelling the quantitative dependence in our density increment then yields a double exponential bound on $N$ in terms of $r$.

In \S\ref{secSchurFF} we give a more detailed exposition of this method for the model problem of Schur's theorem in the finite vector space $\F_2^n$. In \S\ref{general} we generalise the argument to arbitrarily long Brauer configurations over the integers. We improve our bound for four-point Brauer configurations in \S\ref{four point sec}. Finally, in \S\ref{lefmann sec} we show how our methods give comparable bounds for the Ramsey number of certain quadratic equations.  

\subsection*{Notation}

The set of positive integers is denoted by $\N$. Given $x\geqslant 1$, we write $[x]:=\{1,2,...,\floor{x}\}$.  If $f$ and $g$ are functions, and $g$ takes only positive values, then we use the Vinogradov notation $f\ll g$ if there exists an absolute positive constant $C$ such that $|f(x)|\leqslant Cg(x)$ for all $x$. We also write $g\gg f$ or $f=O(g)$ to denote this same property.  The letters $C$ and $c$ are  used to denote absolute constants, whose values may change from line to line. Typically $C$ denotes a large constant $C>1$, whilst $c$ denotes a small constant $0<c<1$.

\subsection*{Acknowledgements}  We thank Tom Sanders for alerting us to the existence of \cite{SandersBootstrapping}, and for his generosity in synchronising release. We also thank Tom Sanders for his thorough and helpful comments. The second author thanks Ben Green for  suggesting this problem.

\section{Schur in the finite field model}\label{secSchurFF}

We illustrate the key ideas of our approach in proving Schur's theorem over $\F_{2}^{n}$. This asserts that, provided the dimension $n$ is sufficiently large relative to the number of colours $r$, any partition $\F_{2}^{n}=A_{1}\cup\cdots\cup A_{r}$ possesses a colour class $A_i$ containing vectors $x,y,z$ with $y \neq 0$ and such that $x+y=z$. The goal of this section is to obtain a quantitative bound on the dimension $n$  in terms of $r$. 

The argument of this section is purely expository, the resulting bound being slightly worse than that given by a standard application of Ramsey's theorem (see \cite[\S3.1]{ramseytheory}) or  Schur's original argument (see \cite{CwalinaSchoen}). We have since learned that the same ideas are discussed in Shkredov \cite[\S5]{Shkredov}.

\begin{theorem}[Schur in the finite field model]\label{model schur theorem}
	Consider a partition of $\F_{2}^{n}$ into $r$ sets $A_{1},...,A_{r}$. If $n$ satisfies
	\begin{equation*}
	n> \sqrt{2}\, r^{3}+\log_2(2r),
	\end{equation*} 
	then there exists $i\in[r]$ and $x,y,z\in A_{i}$ with $y\neq 0$ such that $x+y = z$.
\end{theorem}

\begin{remark}[Distinctness of $x,y,z$]
One can guarantee that the $x,y,z\in A_{i}$ that are obtained are distinct and non-zero by introducing a new partition $\F_{2}^{n}=A_{0}'\cup A_{1}'\cup\cdots\cup A_{r}'$ by setting $A_{0}'=\{0\}$ and $A_{i}':=A_{i}\setminus\{0\}$ for all $i\in[r]$. Applying the above theorem (with $r$ replaced by $r+1$) to this new partition gives distinct non-zero $x,y,z\in A_{i}$ for some $i\in[r]$ satisfying $x+y=z$.
\end{remark}

Inspired by Cwalina--Schoen \cite{CwalinaSchoen}, we deduce Theorem \ref{model schur theorem} from the following dichotomy. This argument is a variant of Sanders's `$99$\% Bogolyubov theorem' \cite{SandersAdditive}, which asserts that the difference set of a dense set contains $99$\% of a subspace of bounded codimension.

\begin{lemma}[Sparsity--expansion dichotomy]\label{lemSchurDich}
	Let $A_1\cup \dots \cup A_r = \F_2^n$ be a partition of $\F_{2}^{n}$ into $r$ parts. Then there exists a subspace $H\leqslant\F_{2}^{n}$ with $\codim(H) \leqslant \sqrt{2}r^3$ such that for any $i \in [r]$ we have one of the two following possibilities.
	\begin{itemize}
		\item (Sparsity).  
		\begin{equation}\label{eqnSchurSparse}
			|A_i \cap H| < \trecip{r}|H|;
		\end{equation}
		\item (Expansion).  
		\begin{equation}\label{eqnSchurExp}
			|(A_i  - A_i) \cap H| \geq \brac{1 - \trecip{2r}} |H|.
		\end{equation}
	\end{itemize} 
\end{lemma}

The idea is that, as one of the colour classes $A_i$ is dense, its difference set $A_i - A_i$ must (by Sanders' result) contain 99\% of a `large' subspace $H$. Were $A_i$ itself to contain more that 1\% of this subspace, then we would be done, since then $(A_i - A_i)\cap A_i \neq \emptyset$ and we would obtain the desired Schur triple $x,y,z\in A_{i}$.  Unfortunately, this cannot always be guaranteed: consider the case in which $A_i$ is a non-trivial coset of a subspace of co-dimension 1.

To overcome this, we run Sanders' proof with respect to all of the colour classes simultaneously, constructing a subspace $H$ which is almost covered by $A_i - A_i$ for \emph{all} $i\in[r]$.  If such a $H$ were obtainable we would be done as before, since (by the pigeonhole principle) some colour class has large density on $H$.  Again, this is slightly too much to hope for, as sets which are `hereditarily sparse' cannot be good candidates for a 99\% Bogolyubov theorem.  Fortunately, such sets can be accounted for in our argument.

Before we proceed to the proof of Lemma \ref{lemSchurDich}, let us use this lemma to prove our finite field model of Schur's theorem.

\begin{proof}[Proof of Theorem \ref{model schur theorem}]
Let $H$ denote the subspace provided by the dichotomy.  By the pigeonhole principle, there exists some $A_i$ satisfying
$$
|A_i\cap H| \geqslant \trecip{r}|H|.
$$
Our assumption on the size of $n$ then implies that
$$
|A_i \cap H| > \trecip{2r}|H| + 1.
$$

Since $A_{i}$ is not sparse on $H$, in the sense of \eqref{eqnSchurSparse}, it must instead satisfy the expansion property \eqref{eqnSchurExp}.  By inclusion--exclusion
$$
|A_i\cap (A_i-A_i)\cap H| \geqslant |A_i \cap H| + |(A_i-A_i)\cap H| - |H| >1.
$$
In particular, the set $A_i\cap (A_i-A_i)$ contains a non-zero element.
\end{proof}

\subsection{A maximal translate increment strategy}

It remains to prove Lemma \ref{lemSchurDich}.  Following Sanders \cite{SandersAdditive}, we accomplish this via density increment.  We cannot merely increment the density of each individual colour class on translates of different subspaces, since our final dichotomy involves a single subspace $H$ which is uniform for all $A_i$.  We therefore have to increment a more subtle notion of density, namely the \emph{maximal translate density}
$$
\Delta_{H}=\Delta_H(A_1, \dots, A_r) := \sum_{i=1}^r \max_x \frac{|A_i \cap (x+H)|}{|H|}.
$$
This is a non-negative quantity bounded above by $r$.  It follows that a procedure passing to subspaces $H_0 \geq H_1 \geq H_2 \geq \dots$, which increments $\Delta_{H_i}$ by a constant amount at each iteration, must terminate in a constant number of steps (depending on $r$).  

We first observe that to increment $\Delta_{H}$ it suffices to find a subspace where one of the colour classes increases their maximal translate density.  To this end, write 
\begin{equation}\label{maximal translate density}
\delta_H(A) := \max_x \frac{|A \cap (x+H)|}{|H|}.
\end{equation}

\begin{lemma}[Maximal translate density is preserved on passing to subspaces]\label{subspace preservation lemma}
Let $H_1 \geq H_2$ be subspaces of $\F_{2}^{n}$.  Then for any $A\subset\F_{2}^{n}$ we have
$$
\delta_{H_2}(A) \geqslant \delta_{H_1}(A).
$$
\end{lemma}

\begin{proof}
As $H_{2}\leq H_{1}$, we can write $H_{1}$ as a disjoint union of cosets of $H_{2}$. This means that we can find $V\subset H_{1}$ such that $H_1 = \sqcup_{y \in V} (y + H_2)$. Hence
$$
  |A \cap (x + H_1)| = \sum_{y \in V} |A \cap (x+y+H_2)| \leqslant |V| \max_z |A \cap (z+H_2)|.
$$
Choosing $x$ so that $A$ has maximal density on $x+H_1$ gives the result.
\end{proof}

We now prove Lemma \ref{lemSchurDich} using a density increment strategy for the maximal translate density. The argument proceeds by showing that if our claimed dichotomy does not hold, then we may pass to a subspace on which the colour classes have larger maximal translate density. 

The process of identifying such a subspace involves the use of Fourier analysis. Given a subspace $H\leq\F_{2}^{n}$ and a function $f:H\to\C$, we define the \emph{Fourier transform} $\hat{f}:\hat{H}\to\C$ of $f$ by
\begin{equation*}
\hat{f}(\gamma):=\sum_{x\in H}f(x)\gamma(x).
\end{equation*}
Here $\hat{H}$ denotes the \emph{dual group} of $H$, which is the group of homomorphisms $\gamma:H\to\C^\times$. Since every element of $H$ has order at most $2$, the value $\gamma(x)$ must be $\pm 1$ for all $x\in H$ and $\gamma\in\hat{H}$.

\begin{proof}[Proof of Lemma \ref{lemSchurDich}]
We proceed by an iterative procedure, at each stage of which we have a subspace $H = H^{(m)} \leq \F_{2}^n$ of codimension $m$ satisfying
$$
\Delta_{H^{(m)}} \geqslant \frac{m}{\sqrt{2}r^2}.
$$
We initiate this procedure on taking $H^{(0)} := \F_{2}^n$.  Since $\Delta_{H^{(m)}} \leqslant r$ this procedure must terminate at some $m \leqslant \sqrt{2}r^3$.

Given $H = H^{(m)}$ we define three types of colour class.
\begin{itemize}
\item  (Sparse colours).  $A_i$ is \emph{sparse} if 
$$
\delta_H(A_i) < \trecip{r};
$$
\item  (Dense expanding colours).  $A_i$ is \emph{dense expanding} if $\delta_H(A_i) \geqslant \trecip{r}$ and we have the expansion estimate
\begin{equation}\label{schur non expansion}
|(A_i  - A_i) \cap H| > \brac{1 - \trecip{2r}} |H|;
\end{equation}
\item  (Dense non-expanding colours).  $A_i$ is \emph{dense non-expanding} if it is neither sparse nor dense expanding.
\end{itemize}
If there are no dense non-expanding colour classes, then the dichotomy claimed in our lemma is satisfied, and we terminate our procedure.  Let us show how the existence of a dense non-expanding colour class $A_i$ allows the iteration to continue.

By the definition of maximal translate density, there exists $t$ such that
$$
|A_i \cap (t + H)| = \delta_H(A_i)|H|.
$$ 
We define dense subsets $A, B\subset H$ by taking
\begin{equation}\label{schur A B defn}
A := (A_i - t)\cap H \quad \text{and} \quad B := H \setminus \bigbrac{A_i - A_i}.
\end{equation}
Writing $\alpha$ and $\beta$ for the respective densities of $A$ and $B$ in $H$, our dense non-expanding assumption implies that $\alpha \geq 1/r$ and $ \beta \geqslant 1/(2r)$.  Moreover, it follows from our construction \eqref{schur A B defn} that 
$$
\sum_{x -x' = y} 1_A(x) 1_A(x') 1_B(y) = 0.
$$
Comparing this to the count
$$
\sum_{x -x' = y} \alpha 1_H(x) 1_A(x') 1_B(y) = \alpha^2\beta |H|^2,
$$
we deduce, on writing $f_{A}:=1_A - \alpha 1_H$, that
$$
\Bigabs{\E_{\gamma \in \hat{H}} \overline{\hat{f}_A(\gamma)} \hat{1}_A(\gamma) \hat{1}_B(\gamma)} = \biggabs{\sum_{x -x' = y} f_A(x) 1_A(x') 1_B(y)} \geqslant \alpha^2 \beta |H|^2.
$$

By Cauchy--Schwarz and Parseval's identity, there exists $\gamma \neq 1_H$ such that
$$
\Bigabs{\sum_{x \in H} f_A(x) \gamma(x)} \geqslant \frac{\alpha^2\beta}{\sqrt{\alpha\beta}} |H| \geqslant \frac{|H|}{\sqrt{2}r^2}.
$$
Partitioning $H$ into level sets of $\gamma$, gives
$$
\Bigabs{\sum_{ \gamma(x) = 1} f_A(x)} + \Bigabs{\sum_{ \gamma(x) = -1} f_A(x)} \geqslant \frac{|H|}{\sqrt{2}r^2}.
$$
Since $f_A$ has mean zero, we deduce that the two terms on the left of the above inequality are equal. This implies that there exists $\omega \in\set{ \pm 1}$ such that
$$
2\sum_{\gamma(x) = \omega} f_A(x) \geqslant \frac{|H|}{\sqrt{2}r^2}.
$$

Observe that, since $\gamma \neq 1_H$, the set $H' := \set{x \in H : \gamma(x) = 1}$ is a subspace of $H$ of codimension $1$. Hence on choosing  $y\in H$ with $\gamma(y)=\omega$ we have 
$$
\frac{|A \cap (y+H')|}{|H'|} \geqslant \alpha + \frac{1}{\sqrt{2}r^2}.
$$
By combining this with Lemma \ref{subspace preservation lemma} and our definition \eqref{schur A B defn} of $A$, we deduce that
$$
\Delta_{H'} \geqslant \Delta_H + \tfrac{1}{\sqrt{2}r^2} \qquad \text{and} \qquad \codim(H') = \codim(H) + 1.
$$
We have therefore established that our iteration may continue, completing the proof of the lemma.
\end{proof}

\section{Brauer configurations over the integers}\label{general}

In this section we use higher order Fourier analysis to study longer Brauer configurations and prove Theorem \ref{weak brauer theorem}. Henceforth, we fix the parameter $k\geqslant 2$ to denote the length of the progression in the Brauer configuration under consideration.  We emphasise that this section streamlines substantially if the reader is only interested in a double exponential bound in the colour aspect, as opposed to the more explicit bound \eqref{mainBound}.

Given finitely supported $f_{1},f_{2},...,f_{k},g:\Z\to\R$ we introduce the counting operator
\begin{equation}\label{counting op}
\Lambda(f_{1},f_{2},...,f_{k};g):=\sum_{d,x\in\Z}f_{1}(x)f_{2}(x+d)\cdots f_{k}(x+(k-1)d)g(d).
\end{equation}
For brevity, write $\Lambda(f;g):=\Lambda(f,f,...,f;g)$. For given finite sets $A,B\subset\N$, the number of arithmetic progressions of length $k$ in $A$ with common difference in $B$ is given by $\Lambda(1_{A};1_{B})$. 

\begin{lemma}\label{propBrauCount}
	Let $M\in\N$ with $M\geqslant k$. If $B\subset[M/(2k-2)]$, then
	\begin{equation*}
	\Lambda(1_{[M]};1_{B})\geqslant\trecip{2}|B|M.
	\end{equation*}
\end{lemma}
\begin{proof}
	Since $B\subset[M/(2k-2)]$, we have $M-(k-1)d\geqslant M/2$ for all $d\in B$. Thus
	\begin{equation*}
	\Lambda(1_{[M]};1_{B})=\sum_{d\in B}(M-(k-1)d)\geqslant \trecip{2}|B|M.
	\end{equation*}
\end{proof}

\subsection{Gowers norms}

Gowers \cite{GowersNewFour} observed that arithmetic progressions of length  four or more are not controlled by ordinary (linear) Fourier analysis. Similarly, four-point Brauer configurations (and longer) require higher order notions of uniformity -- they have \emph{true complexity} greater than $1$ (see \cite{GowersWolfTrue} for further details).  To overcome this difficulty, Gowers introduced a sequence of norms which can be used to measure the higher order uniformity of sets and functions.

\begin{definition}[$U^{d}$ norms]
	Let $f:\Z\to\R$ be a finitely supported function. For each $d\geqslant 2$, the $U^{d}$ \emph{norm} $\lVert f\rVert_{U^{d}}$ of $f$ is defined by
	\begin{equation}\label{eqnFinGowNorm}
	\lVert f\rVert_{U^{d}} :=\left( \sum_{x\in \Z}\sum_{\vh\in \Z^{d}}\,\Delta_{h_1, \dots, h_d}f(x)\right)  ^{1/2^{d}},
	\end{equation}
	where the \emph{difference operators} $\Delta_{h}$ are defined inductively by
	\begin{equation*}
	\Delta_{h}f(x):=f(x)f(x+h)
	\end{equation*}
	and
	\begin{equation*}
	\Delta_{h_1, \dots, h_d}f:=\Delta_{h_{1}}\Delta_{h_{2}}\cdots\Delta_{h_{d}}f.
	\end{equation*}
\end{definition}

\begin{remark}
	In the literature, and in Gowers' original paper, it is common to  work with functions $f:\Z/p\Z\to\R$, defining $\lVert f\rVert_{U^{d}(\Z/p\Z)}$ by summing over $x\in\Z/p\Z$ and $\vh\in(\Z/p\Z)^{d}$ in \eqref{eqnFinGowNorm}. Given a prime $p>N$, one can embed the interval $[N]$ into $\Z/p\Z$ by reduction modulo $p$. This allows us to identify a function $f:[N]\to\R$ with an extension $\tilde{f}:\Z/p\Z\to\R$ on taking $\tilde{f}(x)=0$ for all $x\in(\Z/p\Z)\setminus[N]$. One can observe that if $p>2(d+1) N$, then $\lVert f\rVert_{U^{d}}=\lVert \tilde{f}\rVert_{U^{d}(\Z/p\Z)}$. This is due to the fact that the interval $[N]\subset\Z$ and the embedding of $[N]$ into $\Z/p\Z$ are \emph{Freiman isomorphic of order $d+1$} (see \cite[\S5.3]{TaoVu} for further details).
\end{remark}
We note that if $f:\Z\to[-1,1]$ is supported on $[N]$, then
\begin{equation}\label{norm bounds}
\trecip{2} N^{\frac{d+1}{2^{d}}} \leqslant \norm{f}_{U^d} \leqslant  N^{\frac{d+1}{2^{d}}}.
\end{equation}
The lower bound follows from inductively applying the Cauchy--Schwarz inequality in the form
$$
\norm{f}_{U^{d}} = \brac{\sum_{h_1, \dots, h_{d-1}}\abs{\sum_x \Delta_{h_1, \dots, h_{d-1}} f(x)}^2}^{1/2^d} \geq (2N)^{-\frac{d-1}{2^{d}}}\norm{f}_{U^{d-1}},
$$
the factor of 2 resulting from the observation that if $\supp(f) \subset [N]$, then the $h_i$ in \eqref{eqnFinGowNorm} contribute only if $h_i \in (-N, N)$.
The upper bound is a consequence of the fact that  $\norm{f}_{U^d}^{2^d}$ counts the number of solutions to a system of $2^d-d-1$ independent linear equations in $2^d$ variables, each weighted by $f$.
\begin{definition}[Uniform of degree $d$]\label{uniform def}
We say that $A \subset [N]$ is \emph{$\eps$-uniform of degree $d$} if
\begin{equation*}
\lVert 1_{A} - \E_{[N]}(1_A) 1_{[N]}\rVert_{U^{d+1}}\leqslant \eps\lVert 1_{[N]}\rVert_{U^{d+1}},
\end{equation*}
where $\E_{[N]}(1_A) := |A\cap [N]|/N$ denotes the density of $A$ on $[N]$.
More generally, given $P\subset[N]$, we say that $A$ is \emph{$\eps$-uniform of degree $d$ on $P$} if
\begin{equation*}
\lVert 1_{A} - \E_{P}(1_A) 1_{P}\rVert_{U^{d+1}}\leqslant \eps\lVert 1_{P}\rVert_{U^{d+1}}.
\end{equation*}
\end{definition}
Gowers showed that one can study sets which lack arithmetic progressions of length $k$ by considering their uniformity. If a set has density $\alpha$ in $[N]$ and is $\eps$-uniform of degree $k-2$, for some small $\eps = \eps(k, \alpha)$, then $A$ contains a proportion of $\alpha^k$ of the total progressions of length $k$ in the interval $[N]$.  Hence the only way a uniform set can lack $k$-term progressions is if it has few elements.

A similar result holds for Brauer configurations, see for instance \cite[Appendix C]{GreenTaoLinear}.  In order to avoid the introduction of an (admittedly harmless) absolute constant resulting from the passage to a cyclic group, we give the simple proof.

\begin{lemma}[Generalised von Neumann for $\Lambda$]\label{lemGenVon}
	Let $f_{1},...,f_{k},g:[N]\to[-1,1]$. Then for each $j\in[k]$ we have
	\begin{equation*}
	|\Lambda(f_{1},...,f_{k};g)|\leqslant N^{2}\left(\frac{\lVert f_{j}\rVert_{U^{k}}^{2^{k}}}{N^{k+1}}\right)^{1/2^{k}} \left( \frac{\lVert g\rVert_{U^{k}}^{2^{k}}}{N^{k+1}}\right) ^{1/2^{k}}.
	\end{equation*}
	\end{lemma}

\begin{proof}	We prove the case where $j=k$. The other cases follow on performing a change of variables $x'=x+id$ preceding each application of the Cauchy-Schwarz inequality.
	
	Applying the Cauchy-Schwarz inequality with respect to the $x$ variable shows that $|\Lambda(f_1,...,f_{k};g)|$ is bounded above by
	\begin{align*}
	\left(\sum_{x\in\Z}|f_{1}(x)|^{2}\right)^{1/2}\left(\sum_{x,d,d'\in\Z}g(d)g(d')\prod_{i=1}^{k-1}f_{i+1}(x+id)f_{i+1}(x+id')\right)^{1/2}.
	\end{align*}
	Using the fact that $|f_{1}(x)|\leqslant 1_{[N]}(x)$ holds for all $x\in\Z$, and by performing a change of variables $d'=d+h$, we deduce that
	\begin{equation*}
	|\Lambda(f_{1},...,f_{k};g)|^{2}\leqslant N\sum_{x,h\in\Z}\sum_{d\in\Z}\Delta_{h}g(d)\prod_{i=1}^{k-1}\Delta_{ih}f_{i+1}(x+id).
	\end{equation*}
	By applying the Cauchy-Schwarz inequality a further $k-2$ times, each time with respect to all variables except for $d$, we see that $|\Lambda(f_1,...,f_{k};g)|^{2^{k-1}}$ is bounded above by
	\begin{equation*} N^{2^{k}-k-1}\sum_{\vh\in\Z^{k-1}}\sum_{x\in\Z}\Delta_{(k-1)h_1, (k-2)h_2, \dots, h_{k-1}}f_{k}(x)  \sum_{d\in\Z}\Delta_{h_1, \dots, h_{k-1}}g(d).  
	\end{equation*}
	By applying Cauchy-Schwarz with respect to the $\vh$ variable, the above sum is at most $ S^{1/2}\lVert g\rVert_{U^{k}}^{2^{k-1}}$, where $S$ is equal to
	\begin{equation*}
	\sum_{\vh\in\Z^{k-1}}\left\lvert\sum_{x\in\Z}\Delta_{(k-1)h_1, (k-2)h_2, \dots, h_{k-1}}f_{k}(x)\right\rvert^{2}.
	\end{equation*}
	Since the terms in the sum over $\vh$ are non-negative, we can extend the  summation from $\Z^{k-1}$ to $(k-1)^{-1}\cdot \Z \times (k-2)^{-1}\cdot \Z \times \dots \times \Z$, yielding the lemma.
\end{proof}

\begin{corollary}[$U^{k}$ controls $\Lambda$]\label{corUkCont}
	Let $f_1, f_2,g:[N]\to[0,1]$. Then
	\begin{equation*}
	|\Lambda(f_{1};g)-\Lambda(f_{2};g)|\leqslant kN^{2}\frac{\lVert f_{1}-f_{2}\rVert_{U^{k}}}{N^{(k+1)2^{-k}}}.
	\end{equation*}
\end{corollary}
\begin{proof}
	Observe that $\Lambda(f_{1};g)-\Lambda(f_{2};g)$ can be written as the sum of $k$ terms
	\begin{equation*}
	\Lambda(f_{1}-f_{2},f_{1},...,f_{1};g)+\Lambda(f_{2},f_{1}-f_{2},f_{1},...,f_{1};g)+\cdots+\Lambda(f_{2},...,f_{2},f_{1}-f_{2};g).
	\end{equation*}
	Recall from (\ref{norm bounds}) that $\lVert g\rVert_{U^{k}}^{2^{k}}\leqslant N^{k+1}$. Since $f_{1}-f_{2}$ takes values in $[-1,1]$, the result now follows from the triangle inequality and Lemma \ref{lemGenVon}.
\end{proof}

Lemma \ref{propBrauCount} shows us that, for any non-empty $B\subset[N/(2k-2)]$ and $\alpha>0$, we have
\begin{equation*}
\Lambda(\alpha 1_{[N]};1_{B})\geqslant \tfrac{1}{2}\alpha^{k}|B|N.
\end{equation*}
 Combining this with Corollary \ref{corUkCont} we see that, if $A\subset[N]$ has density $\alpha>0$ and is $\eps$-uniform of degree $k-1$ for some `very small' $\eps>0$, then the difference
\begin{equation*}
|\Lambda(1_{A};1_{B})-\Lambda(\alpha 1_{[N]};1_{B})|
\end{equation*}
is also small. This then implies that $A$ contains an arithmetic progression of length $k$ with common difference in $B$. Hence sets $A$ lacking such arithmetic progression  cannot be uniform. A key observation of Gowers is that this lack of uniformity implies that the set $A$ exhibits significant bias towards a long arithmetic progression inside $[N]$. \\[10pt]
\noindent\textbf{Gowers' density increment lemma.} \emph{Let $d\geqslant 1$ and $0 < \eps \leq \trecip{2}$.  Suppose that $A\subset[N]$ is not $\eps$-uniform of degree $d$ as in Definition \ref{uniform def}.
	Then, on setting 
	\begin{equation}\label{GowersBd}
\eta=\eta(d,\eps) :=  \recip{4}\brac{\frac{\eps}{8(d+2)}}^{2^{d +1+ 2^{d+10}}},
	\end{equation}
	there exists an arithmetic progression $P\subset[N]$ such that
	\begin{equation*}
	|P|\geqslant \eta N^{\eta} \quad\text{and}\quad
	\frac{|A\cap P|}{|P|}\geqslant \frac{|A|}{N}+\eta.
	\end{equation*}
}
\begin{proof}

Let $p$ be a prime in the interval $2(d+2) N < p \leq 4(d+2) N$, so that on setting $f := 1_A - \E_{[N]}(1_A) 1_{[N]}$ and viewing this as a function on $\Z/p\Z$ the lower bound in \eqref{norm bounds} gives 
$$
\sum_{\vh \in (\Z/p\Z)^{d}} \abs{\sum_{x\in\Z/p\Z} \Delta_h f(x)}^2 \geq \brac{ \tfrac{\eps }{8(d+2)}}^{2^{d+1}} p^{d+2}.
$$
Hence, according to Gowers' \cite[p.478]{GowersSzem} definition of $\alpha$-uniformity, $f$ is not $\alpha$-uniform of degree $d$ on $\Z/p\Z$ with 
$$
\alpha = \brac{ \tfrac{\eps }{8(d+2)}}^{2^{d+1}}.
$$
Let $\beta := \alpha^{2^{2^{d+10}}}$. Applying Gowers' local inverse theorem for the $U^{d+1}$-norm \cite[Theorem 18.1]{GowersSzem} there exists a partition of $\Z/p\Z$ into  (integer) arithmetic
progressions $P_1$, \dots, $P_M$ with $M\leq p^{1-\beta}$ and such that 
$$
\sum_{j=1}^M \biggabs{\sum_{x\in P_j} f(x)}  \geq \beta p.
$$
Since $f$ is supported on $[N]$, we may assume that $P_j \subset [N]$ for all $j$.  
As $f$ has mean zero we may apply (the proof of) \cite[Lemma 5.15]{GowersSzem} to obtain a progression $P \subset [N]$ with $|P| \geq \trecip{4}\beta p^\beta$ which also satisfies
$$
\sum_{x\in P} f(x) \geq  \trecip{4}\beta |P|.
$$
\end{proof}

\subsection{Maximal translate density}
As in the previous section, we prove Theorem \ref{weak brauer theorem} by a maximal translate density increment argument.
For $q, M \in \N$ and $A\subset\Z$, define the \emph{maximal translate density }
$$
\delta_{q, M}(A) := \max_{x\in\Z} \frac{|A \cap (x+q\cdot[M])|}{M}.
$$
Given a collection of non-empty subsets $A_{1},...,A_{r}\subset[N]$, we collate their densities into the quantity
$$
\Delta(q, M) = \Delta(q, M;\{A_{i}\}_{i=1}^{r}) :=  \sum_{i=1}^r \delta_{q, M}(A_i).
$$
We write $\Delta(q,M)$ when it is clear from the context which collection of sets $\{A_{i}\}_{i=1}^{r}$ we are working with. 

In the previous section, where we worked with subspaces of $\F_{2}^{n}$, we showed (Lemma \ref{subspace preservation lemma}) that the maximal translate density does not decrease when passing to a subspace. This is no longer true when passing to subprogressions in $\Z$. However, we can still increment $\Delta(q,M)$ if the subprogression we pass to is not too long.

\begin{lemma}[Approximately preserving max translate density]\label{approximately preserving} Given positive integers $M ,M_1, q, q_1$ and a finite set $A\subset\Z$, we have
	\begin{equation*}
	\delta_{qq_{1},M_{1}}(A)\geqslant\delta_{q,M}(A)\left( 1-\tfrac{q_{1}M_{1}}{M}\right) .
	\end{equation*}
\end{lemma}
\begin{proof}
	By definition of $\delta_{q,M}$, we can find $t\in\Z$ such that
	\begin{equation*}
	\delta_{q,M}(A)M=|A\cap (t+q\cdot[M])|.
	\end{equation*}
	Let $\tilde{A}:=A\cap (t+q\cdot[M])$. Note that
	\begin{equation*}
	\delta_{q,M}(\tilde{A})M=|\tilde{A}|=\delta_{q,M}(A)M.
	\end{equation*}
	Now observe that the collection of translates $\{x+qq_{1}\cdot[M_{1}]:x\in\Z \}$ covers $\Z$, and each integer $m\in\Z$ lies in exactly $M_{1}$ such translates. This gives
	\begin{equation}\label{eqnSumTildeC}
	\sum_{x\in\Z}|\tilde{A}\cap (x+qq_{1}\cdot[M_{1}])|=|\tilde{A}|M_{1}=\delta_{q,M}(A)MM_{1}.
	\end{equation}
	Let $\Omega$ be given by
	\begin{equation*}
	\Omega:=\{x\in\Z: \tilde{A}\cap (x+qq_{1}\cdot[M_{1}])\neq\emptyset\}.
	\end{equation*}
	Now suppose $x\in\Omega$. Since $\tilde{A}\subset t+q\cdot[M]$, we can find $u\in[M]$ and $u_{1}\in[M_{1}]$ such that
	$
	x-t=q(u-q_{1}u_{1}).
	$
	From this we see that
	\begin{equation*}
	(x-t)\in[q(1-q_{1}M_{1}),q(M-q_{1})]\cap (q\cdot\Z).
	\end{equation*}
	We therefore deduce that
	$
	|\Omega|\leqslant M+q_{1}M_{1}.
	$
	Applying the pigeonhole principle to (\ref{eqnSumTildeC}), we conclude that
	\begin{equation*}
	\delta_{qq_{1},M_{1}}(A)\geqslant\delta_{qq_{1},M_{1}}(\tilde{A})\geqslant\delta_{q,M}(A)\tfrac{M}{M+q_{1}M_{1}}. 
	\end{equation*}
	This implies the desired bound.
\end{proof}

\begin{corollary}[Subprogression density increment]\label{brauer preservation}
	Let $M,M_{1},q,q_{1}\in\N$, and let $A_{1},...,A_{r}\subset[N]$ be non-empty sets. If $\delta_{qq_1, M_1}(A_i) \geqslant \delta_{q, M}(A_i) + \eta$ for some $i\in[r]$ and some $\eta>0$, then
	$$
	\Delta(qq_1, M_1) \geqslant \Delta(q, M) -\tfrac{q_{1}M_{1}}{M}r+ \eta.
	$$
\end{corollary}

The following lemma allows us to pass to a subprogression whose common difference and length are sufficiently small to allow for an effective employment of Corollary \ref{brauer preservation}. 
\begin{lemma}\label{preserving density}
Let $q$ and $M$ be positive integers with $M \leq 2\floor{N/q}$.  For any $A \subset [N]$ there exists an arithmetic progression $P$ of common difference $q$ and length $|P| \in [\trecip{2}M, M]$ such that
$$
\frac{|A \cap P|}{|P|} \geq \frac{|A\cap [N]|}{N}.
$$
\end{lemma}

\begin{proof}
Let us first give the argument for $q = 1$.  We partition $[N]$ into the intervals
$$
(0, \ceil{M/2}] \cup (\ceil{M/2}, 2\ceil{M/2}] \cup \dots  \cup (m\ceil{M/2}, N],
$$
for some $m$ with $N - m\ceil{M/2} \leq \ceil{M/2}$.  If $N - m\ceil{M/2} = \ceil{M/2}$, then we obtain the result on applying the pigeonhole principle. So we may suppose that $N - m\ceil{M/2} \leq \floor{M/2}$. The pigeonhole principle again gives the result on partitioning similarly, but with the final interval equal to $((m-1)\ceil{M/2}, N]$.

We generalise to $q > 1$ by first partitioning $[N]$ into congruence classes mod $q$. Each such congruence class takes the form $-a + q \cdot [N_a]$ where $0 \leq a < q$ and $ N_a = \floor{(N+a)/q}$.  Since $N_a \geq M/2$ we can use our previous argument to partition $[N_a]$ into intervals, each with length in  $[\trecip{2}M, M]$.  The result follows once again from the pigeonhole principle.
\end{proof}

\subsection{Uniform translates}

We have shown that a highly uniform set contains many Brauer configurations. In general, one cannot guarantee that one of the colour classes in a finite colouring of $[N]$ is uniform. However, we can use Gowers' density increment lemma to show that there exists a long arithmetic progression $q \cdot [M]\subset[N]$ such that each colour class is  uniform on a translate of $q\cdot [M]$, and on the same translate its density is not diminished. 

\begin{lemma}[Uniform maximal translates]\label{thmUniBrau}
Given $0 < \eps \leq \trecip{2}$ and $d \geqslant 1$, let $\eta = \eta(d, \eps)$ denote the constant \eqref{GowersBd} appearing in Gowers' density increment lemma. Suppose that
\begin{equation}\label{C1bd}
N\geqslant\exp\exp(3r\eta^{-2}).
\end{equation}
Then for any sets $A_{1},...,A_{r}\subset[N]$ there exists a homogeneous progression $q\cdot[M]\subset[N]$ with $M\geqslant N^{\exp(-3r\eta^{-2})}$ such that the following is true. For each $i\in[r]$, there exists a translate $a_{i}+q\cdot[M]$ on which $A_{i}$ achieves its maximal translate density and on which $A_{i}$ is $\eps$-uniform of degree $d$.
\end{lemma}

\begin{proof}
	  We give an iterative procedure, at each stage of which we have positive integers $q_n$ and $M_n$ satisfying 
	\begin{equation}\label{iteration conclusion}
	M_n \geqslant (\eta^2/5r)^{1 + \eta + \dots + \eta^{n-1}} N^{\eta^n}
	\quad \text{and}
	\quad
	\Delta(q_n, M_n) \geqslant n\eta/2.
	\end{equation}
	We initiate this on taking $q_0 := 1$ and $M_0 := N$ (the common difference and length of $[N]$).  Since $\Delta(q, M) \leqslant r$ this procedure must terminate at some $n \leqslant 2r\eta^{-1}$.
	
	Suppose that we have iterated $n$ times to give $q=q_n$ and $M =  M_n$.
	If, for each $A_{i}$, there is a translate $a_{i}+q\cdot[M]$ on which $A_{i}$ is $\eps$-uniform and achieves its maximal translate density, then we terminate our procedure. Suppose then that we can find $A_{j}$ which does not have this property. We now give the iteration step of our algorithm. 
	
	By the definition of maximal translate density, there exists $t\in\Z$ such that
	$$
	|A_j \cap (t + q \cdot [M])| = \delta_{q, M}(A_j)M.
	$$ 
	Let $A := \set{y \in  [M] : t+qy \in A_j }$. Since $A_j$ is not $\eps$-uniform of degree $d$ on $t + q\cdot [M]$, we see that $A$ is not $\eps$-uniform of degree $d$ on $[M]$. 
	By Gowers' density increment lemma, we deduce the existence of a progression $P\subset[M]$ of length $|P| \geqslant \eta M^{\eta}$ such that
	$$
	|A \cap P | \geqslant \left(\delta_{q, M}(A_j) + \eta \right)|P|.
	$$ 
	
	We would like the length and common difference of $P$ to be sufficiently small to allow for the effective employment of Corollary \ref{brauer preservation}.  Using \eqref{C1bd} and \eqref{iteration conclusion} one can verify that $\eta M^\eta \geq 2r\eta^{-1}$, so that the integer $ \lfloor \eta |P|/(2r) \rfloor$ is  positive.   Lemma \ref{preserving density} then gives a  subprogression $x + q'\cdot [M'] \subset P$, of the same common difference as $P$, such that $\trecip{2}\lfloor \eta |P|/(2r) \rfloor \leq M' \leq \lfloor \eta |P|/(2r) \rfloor$ and for some $x$ we have
	\begin{equation*}
	\frac{|A\cap (x+q'\cdot[M'])|}{M'}\geqslant\frac{|A\cap P|}{|P|}.
	\end{equation*} 
	Note that, since $P \subset [M]$ has common difference $q'$ we have $q'|P| \leq M$ and so $q' M' \leq q' |P| \eta/ (2r) \leq M \eta/(2r)$. Hence by Corollary \ref{brauer preservation} we obtain
	$$
	\Delta(q'q, M') \geqslant \Delta(q, M) + \trecip{2}\eta.
	$$

Again using \eqref{iteration conclusion} and \eqref{C1bd} one can check that $M' \geq (\eta^2/5r) M^\eta$, so we obtain \eqref{iteration conclusion} with $(q_{n+1}, M_{n+1}) := (q', M')$, and our iteration can continue.  Taking this iteration through to completion gives the lemma.
\end{proof}
We are now in a position to derive our main theorem.
\begin{proof}[Proof of Theorem \ref{weak brauer theorem}]
	 Let $\eta = \eta(k-1, \eps)$ be the quantity given by \eqref{GowersBd} in Gowers' density increment lemma, with $\eps $ to be determined, and suppose that $N\geqslant\exp\exp(4r\eta^{-2})$. Let $M,q$ be the positive integers obtained by applying Lemma \ref{thmUniBrau} to the partition $[N]=A_{1}\cup\cdots\cup A_{r}$. By the pigeonhole principle, there exists $j\in[r]$ such that
	\begin{equation}\label{abBound}
	|A_{j}\cap q\cdot[M/(2k-2)]|\geqslant\frac{1}{r}\left\lfloor\frac{M}{2(k-1)}\right\rfloor >\frac{M}{4(k-1)r},
	\end{equation}
	the latter following from the fact that 
	$$
	M \geq \exp\brac{\exp(-3r\eta^{-2})\log N} \geq \exp\exp(r\eta^{-2}).
	$$ 
	
	Let $t\in\Z$ be such that 
	\begin{equation*}
	|A_{j}\cap (t+q\cdot [M])|=\delta_{q, M}(A_{j})M.
	\end{equation*}
	We now construct sets $A\subset[M]$ and $B\subset[M/(2k-2)]$ by taking
	\begin{equation*}
	A := \{y \in  [M] : t+qy \in A_j \} \quad \text{and}\quad
	B := \{d \in  [M/(2k-2)] : qd \in A_{j}\}.
	\end{equation*}
	Our goal is to show that $\Lambda(1_{A};1_{B})>0$. If this is the case, then there exists an arithmetic progression of length $k$ in $A$ whose common difference lies in $B$. We can then infer from our construction of $A$ and $B$ the existence of a $(k+1)$-point Brauer configuration in $A_{j}$.
	
	Let $\alpha$ and $\beta$ denote the respective densities of $A$ and $B$ in $[M]$. From the bound \eqref{abBound} it follows that $\alpha,\beta >(4r(k-1))^{-1}$.
	Combining this with Lemma \ref{propBrauCount} gives
	\begin{equation*}
	\Lambda(\alpha 1_{[M]};1_{B})\geqslant\tfrac{1}{2}\alpha^{k}\beta M^{2}>\frac{M^{2}}{2^{2k+3}(k-1)^{k+1}r^{k+1}}.
	\end{equation*}
	Recall that the conclusion of Lemma \ref{thmUniBrau} guarantees that $A$ is $\eps$-uniform of degree $(k-1)$ (as a subset of $[M]$). Applying Corollary \ref{corUkCont} and the upper bound in \eqref{norm bounds} gives
	\begin{equation*}
		|\Lambda(1_{A};1_{B})-\Lambda(\alpha 1_{[M]};1_{B})|\leqslant kM^{2}\eps.
	\end{equation*}
	We therefore obtain $\Lambda(1_{A};1_{B})>0$, and hence the theorem, on taking
	\begin{equation}\label{eps choice}
	\eps^{-1} := k2^{2k+3}(k-1)^{k+1}r^{k+1}.
	\end{equation}
	Since we are assuming that $N \geqslant\exp\exp(4r\eta^{-2})$, this choice of $\eps$ gives a double exponential bound in $r^{O_k(1)}$.  
	
	Finally, we show that the more precise bound \eqref{mainBound} suffices. The inequality $\exp\exp(x) \leq 2^{2^{2x}}$ is valid for all $x \geqslant 1$, so that
	$$
	\exp\exp(4r \eta^{-2}) \leq 2^{2^{8r\eta^{-2}}}.
	$$
	For our choice \eqref{eps choice} of $\eps$, we have
	$$
	\eta^{-1} \leq  \brac{16(k+1)k(k-1)^{k+1} 2^{2k+3} r^{k+1}}^{2^{k + 2^{k+9}}}
	$$
	One may check that
	$$
	16(k+1)k(k-1)^{k+1} \leq 2^{k^2 + 4},
	$$
	so that, on using $r\geq 2$, we have
	$$
	8r\eta^{-2} \leq 8r \brac{ 2^{k^2 + 2k+7} r^{k+1}}^{2^{1+k + 2^{k+9}}} \leq  r^{\brac{k+3}^22^{1+k + 2^{k+9}}} \leq r^{2^{2^{k+10}}},
	$$
	as required.
\end{proof}

\section{An improved bound for four-point configurations}\label{four point sec}

In this section we focus on the four-point Brauer configuration \eqref{four point},   improving our bound on the Ramsey number to $\exp \exp(r^{1+o(1)})$.  Instead of finessing the quantitative aspect of Lemma \ref{thmUniBrau}, we opt to mimic the sparsity--expansion dichotomy of \S\ref{secSchurFF}. This requires a higher order analogue of the difference set $A - A$, namely
$$
\mathrm{Step}_3(A) := \set{d : A\cap (A-d) \cap (A-2d) \neq \emptyset}.
$$

\begin{lemma}[Sparsity--expansion dichotomy]\label{brauer dichotomy lemma}
There exists an absolute constant $C$ such that for $N \geq \exp \exp (Cr\log^2 r)$ the following holds.   For any $r$-colouring $A_1\cup \dots \cup A_r = [N]$ there exist positive integers $q$ and $M$ with $qM\leqslant N$ such that for any $i \in [r]$ we have one of the two following possibilities.
\begin{itemize}
\item (Sparsity).  
\begin{equation}\label{brauer sparsity}
|A_i \cap q \cdot [3M]| < \trecip{r}M;
\end{equation}
\item (Expansion).  
\begin{equation}\label{brauer expansion}
|\mathrm{Step}_3(A_i) \cap q \cdot [M]| > \brac{1 - \trecip{r}} M.
\end{equation}
\end{itemize} 
\end{lemma}

Before proving this, let us first use it to obtain a bound on the Ramsey number of four-point Brauer configurations.
\begin{proof}[Proof of Theorem \ref{improved bound}]
Let $q$ and $M$ denote the numbers provided by the dichotomy.  By the pigeonhole principle, there exists a colour class $A_i$ satisfying
$$
|A_i\cap q \cdot [M]| \geq\trecip{r}M.
$$
So $A_i$ is not sparse in the sense of \eqref{brauer sparsity}.  It follows that $A_i$ must instead satisfy the expansion property \eqref{brauer expansion}.  By inclusion--exclusion
$$
|A_i\cap \mathrm{Step}_3(A_i)\cap q\cdot [M]| \\
\geq |A_i \cap q\cdot [M]| + |\mathrm{Step}_3(A_i)\cap q\cdot [M]| - M > 0.
$$
In particular $A_i\cap \mathrm{Step}_3(A_i)$ contains a non-zero element.
\end{proof}

\begin{proof}[Proof of Lemma \ref{brauer dichotomy lemma}]
If $r=1$, then $[N]=A_{1}$ and so (\ref{brauer expansion}) holds with $q=1$ and $M=N$ for all $N\geqslant\exp\exp(0)>2$. We may therefore henceforth assume that $r \geq 2$. Set $q_0 := 1$ and $N_0 := N$.  We proceed by an iterative procedure, at each stage of which we have positive integers $q_n$, $N_n$ and $M_i = M_i^{(n)}$ ($1\leq i \leq r)$ such that for $n\geq 1$ we have: \\
\begin{enumerate}[(i)]
\item\label{length} $\displaystyle
N_{n} \geq  \exp\brac{-r^{O(1)}}N_{n-1}^{r^{-O(1)}}$;\\ 
\item\label{order} $\displaystyle
M_i \in [\trecip{2}N_n, N_n]$  for all $1\leq i \leq r$;\\
\item\label{preservation} $\displaystyle \delta_{q_n,M_i^{(n)}}(A_i)  \geq  \delta_{q_{n-1} , M_{i}^{(n-1)}}(A_i )$ for all $1\leq i \leq r$;\\
\item\label{density} $ \displaystyle \delta_{q_n,M_i^{(n)}}(A_i)  \geq  \brac{1+c}\delta_{q_{n-1},M_{i}^{(n-1)}}(A_i)$ for some $ i $ with\\[10pt] $\delta_{q_{n-1},M_{i}^{(n-1)}}(A_i) \geq \trecip{7r}$.  Here $c = \Omega(1)$ is an absolute constant.\\
\end{enumerate}

At stage $n$ of the iteration we classify each colour class $A_i$ according to which of the following hold.
\begin{itemize}
\item  (Sparse colours).  These are the colours $A_i$ for which 
$$
\delta_{q_n, M_i}(A_i) < \trecip{7r}. 
$$
\item  (Dense expanding colours).  These are the dense colours $\delta_{q_n, M_i}(A_i) \geq \trecip{7r}$ for which we have the additional expansion estimate
\begin{equation}\label{brauer non expansion}
|\mathrm{Step}_3(A_i) \cap q_n\cdot [N_n/6]| > \brac{1 - \trecip{r}}\floor{ N_n/6}.
\end{equation}
\item  (Dense non-expanding colours).  The class $A_i$ is \emph{dense non-expanding} if it is neither sparse nor dense expanding.
\end{itemize}
If there are no dense non-expanding colours, then we terminate our procedure.  If $N_n < 100$, then we also terminate our procedure.  Let us therefore suppose that $N_n \geq 100$ and there exists a dense non-expanding colour class $A_i$.  Our aim is to show how, under these circumstances, the iteration may continue.

By the definition of maximal translate density, there exists $a$ such that
\begin{equation}\label{max density on Ai}
|A_i \cap (a + q_n \cdot [M_i])| M_i^{-1} = \delta_{q_n, M_i}(A_i) \geq \trecip{7r}.
\end{equation}
Writing $M := M_i$, we define dense subsets $A, B \subset [M]$ by taking
\begin{equation}\label{brauer A B defn}
A := \set{y \in  [M] : a+q_ny \in A_i } \quad \text{and} \quad B := \set{d \in  [N_n/6] : q_nd \notin \mathrm{Step}_3(A_i)}.
\end{equation}
We recall that $M/3 \geq N_n/6 \geq M/6$ .

Letting $\alpha$ denote the density of $A$ in $[M]$, we see that $\alpha$ is equal to the left-hand side of \eqref{max density on Ai}.  Our assumption that $A_i$ is dense non-expanding and $N_n \geq 100$ together imply that  $B$ has size $\gg r^{-1} M$ and that 
$$
\sum_{x, d} 1_A(x) 1_A(x+d)1_A(x+2d) 1_B(d) = 0.
$$
Using the notation \eqref{counting op} and (the proof of) Lemma \ref{propBrauCount} we have 
$$
\abs{\Lambda(1_A; 1_B)  -\Lambda(\alpha 1_{[M]}; 1_B)} \gg \alpha^{3} M|B| \gg r^{-4} M^2.
$$

From hereon, we assume that the reader is familiar with the notation and terminology of Green and Tao \cite{GT09}. Applying \cite[Theorem 5.6]{GT09} in conjunction with Corollary \ref{corUkCont} we obtain a quadratic factor $(\mathcal{B}_1, \mathcal{B}_2)$ of complexity and resolution $\ll r^{O(1)}$  such that the function $f:= \E(1_A \mid  \mathcal{B}_2)$ satisfies
$$
\abs{\Lambda(f;1_B)  -\Lambda(\alpha 1_{[M]}; 1_B)} \gg \alpha^{3} M|B|.
$$
Define the $\mathcal{B}_2$-measurable set
$$
\Omega:= \set{x \in [M] : f(x) \geq (1+c)\alpha},
$$
where $c>0$ is small enough to make the following argument valid.\footnote{Specifically, $c$ is chosen to be sufficiently small relative to all the implicit constants appearing in the inequalities preceding (\ref{omega lower bound}).}

For functions $f_1, f_2, f_3 : [M] \to\R$ we have the bound
\begin{equation}\label{L1 bound}
\abs{\Lambda( f_1, f_2, f_3; 1_B)} \leqslant M|B|\lVert f_{i}\rVert_{L^{1}([M])} \prod_{j \neq i} \norm{f_j}_\infty.
\end{equation}
Invoking the telescoping identity we used to prove Corollary \ref{corUkCont} gives us the bound $\abs{\Lambda(f;1_B)  -\Lambda(f1_{\Omega^c};1_B)} \ll|\Omega||B|$, so that
$$
|\Omega||B| + \abs{\Lambda(f1_{\Omega^c};1_B)  -\Lambda(\alpha1_{[M]};1_B)} \gg  \alpha^{3} M|B| .
$$
Another application of the telescoping identity in conjunction with \eqref{L1 bound} gives 
\begin{align*}
\abs{\Lambda(f1_{\Omega^c};1_B)  -\Lambda(\alpha 1_{[M]};1_B)} & \ll \alpha^2 M|B|
\norm{f1_{\Omega^c}- \alpha1_{[M]}}_{L^{1}([M])}\\
& \ll \alpha^2 M|B|\norm{f - \alpha1_{[M]}}_{L^{1}([M])} + |B||\Omega|,
\end{align*}
so that
$$
 |\Omega|+ \alpha^2 M\norm{f - \alpha1_{[M]}}_{L^{1}([M])}  \gg \alpha^3 M.
$$

Since $f - \alpha1_{[M]}$ has mean zero, its $L^1$-norm is equal to twice the mean of its positive part. The function $\brac{f - \alpha1_{[M]}}_+$ can only exceed $c \alpha$ on $\Omega$, so taking $c$ small enough gives 
\begin{equation}\label{omega lower bound}
|\Omega| \gg \alpha^3M \gg r^{-3} M.
\end{equation}
As $\mathcal{B}_2$ has complexity and resolution $r^{O(1)}$ it contains at most $\exp(r^{O(1)})$ atoms.  By \cite[Proposition 6.2]{GT09} each such atom can be partitioned into a further
\begin{equation}\label{APno}
\exp(r^{O(1)}) M^{1-r^{-O(1)}}
\end{equation}
disjoint arithmetic progressions. Hence $\Omega$ itself can be partitioned into arithmetic progressions, the number of which is at most \eqref{APno}.  Combining this with \eqref{omega lower bound} and \cite[Lemma 6.1]{GT09}, we see that there exists an arithmetic progression $P$ of length at least
$$
\exp\brac{-r^{O(1)}} M^{1/r^{O(1)}}
$$
on which $A$ has density at least $(1+\tfrac{c}{2})\alpha$. By partitioning $P$ into two pieces and applying the pigeon-hole principle, we may further assume that $|P|q\leqslant M$, where $q$ is the common difference of $P$.

Writing $q_{n+1}$ and $N_{n+1} = M_i^{(n+1)}$ for the common difference and length of $P$, we see that \eqref{length} and \eqref{density} are satisfied.  For all $j \in [r]\setminus\set{i}$ we have 
$$
q_{n+1} N_{n+1} \leq M =  M_{i}^{(n)} \leq N_n \leq 2M_{j}^{(n)}.
$$
Hence we may apply Lemma \ref{preserving density} to each colour class $A_j$ with $j \neq i$ to obtain a progression $P_j$ of common difference $q_n$ and length $M_{j}^{(n+1)}$ such that \eqref{order} and \eqref{preservation} hold.  It follows that our iteration may continue.

Our iterative procedure must terminate at stage $n$ for some $n \ll r^{2}$.  To see this, note that the sum of the maximal translate densities $\delta_{q_n ,M_i}(A_i)$ is at most $r$, and this quantity increases by at least $\Omega(1/r)$ at each iteration.   Our next task is to improve this upper bound on the number of iterations to $n\ll r \log r$.

Let $A_i$ denote the colour class for which the density increment \eqref{density} occurs most often.  By the pigeonhole principle this happens on at least $n/r$ occasions. If the density of $A_i$ increments at least $c^{-1}$ times, then its density doubles.  After a further $\trecip{2} c^{-1}$ increments the density of $A_i$ quadruples.  The density of $A_i$ has therefore increased by a factor of $2^m$ if the number of iterations is at least
\begin{equation}\label{no A_i increments}
\ceil{c^{-1}} + \ceil{\trecip{2}c^{-1}} + \dots +  \ceil{\trecip{2^{m-1}}c^{-1}} \leq m + 2c^{-1}.
\end{equation}
The first time $A_i$ increments its initial density is at least $1/(7r)$, so if the number of increments experienced by $A_i$ is at least \eqref{no A_i increments} then its final density is at least $2^m /(7r)$. If $n/r > 2c^{-1} + \ceil{\log_2(7r)}$ then we obtain a density exceeding 1, a contradiction.  It follows that the total number of iterations $n$ satisfies $n = O( r \log_2 r)$.

Having shown that our iteration must terminate in $ n=O(r \log r)$ steps, let us now ensure that termination results from a lack of dense non-expanding colours.  This follows if we can ensure that $N_n \geq 100$.   Applying the lower bound \eqref{length} iteratively we obtain
$$
N_n \geq \exp\brac{-r^{O(1)}} N^{r^{-O(n)}}.
$$ 
Using the fact that $n \ll r\log r$, the right-hand side above is at least 100 provided it is not the case that $N \leq \exp\exp\brac{O(r\log^2 r)}$.  Given this assumption, we obtain the conclusion of Lemma \ref{brauer dichotomy lemma} on taking $M := \floor{N_n/6}$.
\end{proof}

\section{Lefmann quadrics}\label{lefmann sec}

In this section we show how our results can be used to obtain bounds on the Ramsey numbers for quadric equations of the form
\begin{equation}\label{lefquad}
\sum_{i=1}^{s}a_{i}x_{i}^{2}=0.
\end{equation}
Lefmann \cite[Fact 2.8]{lefmann} established the following sufficient condition for equations of the above form to be partition regular.\\

\noindent\textbf{Lefmann's criterion.} \emph{Suppose that $a_{1},...,a_{s}\in\Z\setminus\{0\}$ satisfy the following two properties:
	\begin{enumerate}[\upshape (i)]
		\item there exists a non-empty set $I\subset[s]$ such that $\sum_{i\in I}a_{i}=0$;
		\item there exists a positive integer $\lambda$ such that the system
		\begin{equation*}
		\lambda^{2}\sum_{i\notin I}a_{i}+ \sum_{i\in I}a_{i}u_{i}^{2}=
		\sum_{i\in I}a_{i}u_{i}=0.
		\end{equation*}
		has a solution in integers $(u_i)_{i \in I}$.
	\end{enumerate}
Then in any finite partition of the positive integers $\N=A_{1}\cup\cdots\cup A_{r}$, there exists $i\in[r]$ and $x_{1},...,x_{s}\in A_{i}$ satisfying (\ref{lefquad}).\\
}

Lefmann observed that assumption (i) is necessary for the conclusion of the above theorem to hold. Lefmann then showed that if (i) and (ii) both hold, then \eqref{lefquad} has a solution over any set of the form $\{x,x+d,...,x+(k-1)d,\lambda d\}$, where $k=1+2\max_{i\in I}|u_{i}|$. We  therefore obtain our quantitative version of Lefmann's result (Theorem \ref{intro lefmann}) from the following analogue of Theorem \ref{weak brauer theorem}.

\begin{theorem}[Ramsey bound for generalised Brauer configurations]
	For positive integers $k,\lambda$ there exists an absolute constant $C = C(k, \lambda)$ such that for any $r \geq2$ and $N \geq \exp\exp(r^C)$, if $[N]$ is $r$-coloured then there exists a monochromatic configuration of the form $\{x,x+d,...,x+(k-1)d,\lambda d\}$.
\end{theorem}
\begin{proof}
	This is essentially the same as the proof of Theorem \ref{weak brauer theorem}. 
\end{proof}

\end{document}